\documentclass[12pt,a4paper]{article}
\usepackage{amssymb}
\usepackage{subfigure}
\usepackage{epsf,epsfig,amsfonts,amsgen,amsmath,amstext,amsbsy,amsopn,amsthm}
\usepackage{ifthen}     
\usepackage{calc}       
\usepackage{indentfirst}
\usepackage{graphicx}
\usepackage{graphics,multicol}
\usepackage{color}
\usepackage[colorlinks=true,anchorcolor=blue,filecolor=blue,linkcolor=blue,urlcolor=blue,citecolor=blue]{hyperref}
\usepackage{floatrow}
\usepackage{enumerate}
\usepackage{tikz}
\usepgflibrary{bbox}
\usepackage{enumitem}
\usetikzlibrary{calc}
\usepackage{endnotes}
\usepackage{geometry}
\newtheorem{thm}{Theorem}[section]

\newtheorem{lem}{Lemma}[section]
\newtheorem{cor}{Corollary}[section]

\newtheorem{pro}{Proposition}[section]

\theoremstyle{definition}

\makeatletter
\renewenvironment{proof}[1][\proofname]{%
  \par\pushQED{\qed}%
  \normalfont \topsep6\p@\@plus6\p@ \trivlist
  \item[\hskip\labelsep \textbf{#1}]%
}{%
  \popQED\endtrivlist\@endpefalse
}
\makeatother

\makeatletter
\@addtoreset{case}{lem}
\makeatother

\makeatletter
\@addtoreset{claim}{section}
\makeatother

\allowdisplaybreaks\allowdisplaybreaks[4]
\renewcommand\proofname{\bf Proof}
\addtocounter{section}{0}

\makeatletter
\renewcommand\subsection{\@startsection{subsection}{2}{\z@}%
                                      {-5ex \@plus -1ex \@minus -.2ex}
                                      {2.5ex \@plus .2ex}
                                      {\normalfont\itshape}}
\makeatother

\usetikzlibrary{backgrounds}
\usetikzlibrary{arrows}
\pgfdeclarelayer{edgelayer}
\pgfdeclarelayer{nodelayer}
\pgfsetlayers{background,edgelayer,nodelayer,main}
\tikzstyle{none}=[inner sep=0mm]
\tikzstyle{bluenode}=[fill=blue, draw=black, shape=circle, minimum
size=0.2cm,
inner sep=0pt]
\tikzstyle{whitenode}=[fill={rgb,255: red,245; green,245;
blue,245}, draw=black, shape=circle, minimum size=0.2cm, inner
sep=1pt]
\tikzstyle{yellownode}=[fill=yellow, draw=black, shape=circle, minimum size=0cm, inner sep=1pt]
\tikzstyle{pinknode}=[fill={rgb,255: red,255; green,191; blue,191}, draw=black, shape=circle, minimum size=0cm, inner sep=1pt]
\tikzstyle{blacknode}=[fill=black, draw=black, shape=circle, minimum size=0.2cm, inner sep=0pt]
\tikzstyle{rednode}=[fill={rgb,255: red,244; green,0; blue,0}, draw=black,
shape=circle, minimum size=0.2cm, inner sep=0.1pt]
\tikzstyle{square}=[draw=black, shape=rectangle, minimum
size=0.1cm, inner sep=3pt,fill={rgb,255: red,245; green,245;
blue,245},scale=0.5]
\tikzstyle{dot}=[fill=black, draw=black, shape=circle, minimum size=0.04cm, inner sep=0pt]
\tikzstyle{blackedge}=[line width=1.2pt, black]
\tikzstyle{blackedge_thick}=[-, draw=black, thick,
fill=none]
\tikzstyle{rededge}=[-, draw=red]
\tikzstyle{rededge_thick}=[-, line width=0.45mm, draw=red]
\tikzstyle{blackedge_opacity}=[-, -, draw={rgb,255: red,91; green,87; blue,84},
fill=none, line width=0.4mm, opacity=0.7]
\tikzstyle{balck_dash}=[-, dash pattern=on 0.2mm off 0.2mm]
\tikzstyle{blue_thick}=[-, line width=0.5mm, draw=blue]
\tikzstyle{blueedge}=[-, line width=1pt, blue, opacity=0.7]
\tikzstyle{shadow_silver2}=[-, draw=black, fill={rgb,255: red,186; green,186;
blue,186}, thick]
\tikzstyle{shadow_}=[-, fill={rgb,255: red,186; green,186; blue,186}, draw=none]
\tikzstyle{dashpure_thick}=[-, dashed, line width=0.5mm, draw={rgb,255:
red,128; green,0; blue,128}]
\tikzstyle{pureedge}=[-, draw={rgb,255: red,128; green,0; blue,128},line width=0.2mm]
\tikzstyle{red_dash}=[-, draw=red, dash pattern=on 0.2mm off 0.2mm,line width=0.6mm]

\let\oldbibliography\thebibliography
\renewcommand{\thebibliography}[1]{%
  \oldbibliography{#1}%
  \small
  \setlength{\itemsep}{-3pt}%
  \setlength{\baselineskip}{10pt}
  \setlength{\lineskiplimit}{-\maxdimen}
}
\usepackage{titlesec}

\titleformat{\subsection}        
  {\bfseries\upshape\normalsize} 
  {\thesubsection}               
  {1em}                          
  {}                             
\begin{document}
\title{\bf \Large The sharp upper bounds on the maximum degree and vertex-connectivity of claw-free 1-planar graphs\footnote{The work was supported by the National
Natural Science Foundation of China (Grant No. 12271157, 12371346) and the Postdoctoral Science Foundation of China (Grant  No. 2024M760867).}
\author{ 
{  Licheng Zhang$^{a}$} \thanks{lczhangmath@163.com}, {  Zhangdong Ouyang $^{b}$ \thanks{Corresponding author: oymath@163.com}}, {  Yuanqiu Huang $^{c}$}\thanks{hyqq@hunnu.edu.cn}\\
\small $^{a}$ School of Mathematics, Hunan University, 
\small Changsha, 410082, China\\
\small $^{b}$College of Mathematics and Statistics, Hunan First Normal University, 
\small Changsha, 410205, China \\
\small $^{c}$ College of Mathematics and Statistics, Hunan Normal University, 
\small Changsha, 410081, China 
}}

\date{}
\maketitle
\begin{center}
  {\large\bfseries Abstract}
\end{center}
This paper studies structural properties of claw-free 1-planar graphs, extending earlier work on claw-free planar graphs initiated by Plummer. We establish upper bounds on the maximum degree and the vertex-connectivity of claw-free 1-planar graphs. First, every claw-free 1-planar graph has maximum degree at most 10. In addition, if the graph is 6-connected, then its maximum degree is at most 8. Finally, the vertex-connectivity of any claw-free 1-planar graph is at most 6, and thus every 7-connected 1-planar graph  contains an induced claw. All of these bounds are sharp.



%
\section{Introduction}\label{se-1}

We consider finite simple graphs and use standard terminology and notation from  \cite{Bondy}.  Let $G$ be a graph with
{\it a vertex set} $V(G)$ and {\it an edge set} $E(G)$. The \emph{order} and \emph{size}  of $G$ are $|V(G)|$ and $|E(G)|$, respectively. For brevity, we write $n(G)$ instead of $|V(G)|$ and $e(G)$ instead of $|E(G)|$.
The degree of a vertex $v$ in $G$ is denoted by $d_G(v)$. The maximum degree of $G$ is denoted by  $\Delta(G)$.  The vertex-connectivity (for short, connectivity) of   $G$ is denoted by $\kappa(G)$. 
A subgraph $H$ of a graph $G$ is called an \emph{induced subgraph} of $G$ if any edge in $G$ that joins a pair of vertices in $H$ is also in $H$. The complete bipartite graph $K_{1,3}$ is  called  a {\it claw}.  For a graph $H$, a graph $G$ is {\it $H$-free } if $G$ does not have an induced subgraph isomorphic to $H$. In particular, a graph $G$ is {\it claw-free} if $G$ does not have an induced subgraph isomorphic to a claw. 
A \emph{line graph} $L(G)$ of a graph $G$ is obtained by associating a vertex with each edge of $G$ and connecting two vertices with an edge if the corresponding edges of $G$ have a vertex in common.

Claw-free graphs have been extensively studied for a long time.  In particular, in 1970, Beineke \cite{Beineke} characterized line graphs using nine forbidden induced subgraphs, including the induced claw, leading to the initial study of claw-free graphs as a generalization of line graphs. 
Furthermore, the study of claw-free graphs has been driven by additional motivations, notably in topics such as matchings \cite{Plummer1994} and hamiltonicity \cite{Matthews–Sumner}. For further details, the reader is referred to the classic survey \cite{Faudree}. In this paper, we focus on properties of  claw-free  graphs with special drawings. A graph is {\it planar} if it can be drawn on the plane  such that no edges cross each other.  A \emph{maximal planar graph}  of order at least 3 is a graph that can be embedded in the plane such that every face of the graph is a triangle. The study of claw-free planar graphs was initiated by Plummer in 1989 \cite{Plummer}. Plummer \cite{Plummer} characterized claw-free maximal planar graphs, showing that each of these graphs belongs to an infinite family that can be described recursively. Later, Plummer  \cite{Plummer1994} showed that there is precisely one
3-connected claw-free planar graph of order even which is 2-extendable -- the icosahedron.  Shan, Liang and Kang  \cite{Shan2014} showed that  every claw-free planar graph, different from an odd cycle, is 2-clique-colorable. Notably,
Plummer in \cite{Plummer} showed that every 3-connected claw-free planar graph has maximum degree at most 6 and the bound is sharp. In fact, using Plummer's proof technique, the upper bound 6 still holds for all claw-free planar graphs (without the 3-connectivity condition). The upper bound of the maximum degree of claw-free planar graph is vital for the characterization of claw-free maximal planar graphs and has found broad applications in areas such as matching extendability \cite{Plummer1994}, total coloring \cite{Liang2022}, clique-perfectness \cite{Liang2016}, and clique-transversal numbers \cite{Shan2014}.
Moreover, Plummer  \cite{Plummer} showed that 3-connected claw-free planar graphs without separating triangle (i.e. a cycle of length 3) \footnote{A triangle of a planar graph is separating  if its removal separates the graph.} have maximum degree at most 5.  This directly implies that every 4-connected claw-free planar graph $G$ has $\Delta(G) \leq 5$. It is well-known that every planar graph has connectivity at most 5. Moreover, Corollary 2.2 in the same article \cite{Plummer} implicitly states that there exists exactly one 5-connected claw-free maximal planar graph, namely the Icosahedral graph. In 2014,  Shan, Liang and Kang~\cite{Shan2014} (Lemma~8) proved that if $G$ is a claw-free planar graph without 4-cliques, then $\Delta(G) \leq 5$, and for every vertex $v$ of degree 5, the subgraph induced by its closed neighborhood is a wheel of order 6. As a consequence, every 5-connected claw-free planar graph is maximal planar. Combining these results yields that every  claw-free planar graph has connectivity at most four  with the exception of the Icosahedral graph.

This paper studies structural properties of claw-free 1-planar graphs, extending earlier work on claw-free planar graphs initiated by Plummer.
A graph is  {\it $1$-planar}  if it can be drawn on the plane such that each edge is crossed at most once. A \emph{1-plane graph} (resp. \emph{plane graph}) is a fixed 1-planar drawing (resp. \emph{planar drawing}) of a 1-planar graph (resp. planar graph). 1-planar graphs were first studied by Ringel (1965) \cite{Ringel1965}.  Since then, 1-planar graphs have witnessed substantial progress; see \cite{kobourov2017} for a recent survey and the references therein. It is worth noting that the study of forbidden (induced) subgraphs has also received attention in  1-planar graphs. For example, Chen et al. \cite{Chen} proved that every triangle-free 1-planar graph $G$ has an acyclic edge coloring with $\Delta(G)+16$ colors. Recently, Bekos et al. \cite{Bekos} showed every triangle-free 1-planar graph of order $n \geq 4$  has size at most $3n-6$. However, to the best of our knowledge, the study of claw-free 1-planar graphs is entirely new. First, we  establish an upper bound on the maximum degree of  claw-free 1-planar graphs.

\begin{thm}\label{main0}
Let $G$ be  a $1$-planar graph. If $G$ is claw-free, then $\Delta(G)\le 10 $, and the bound is sharp.
\end{thm}

Furthermore, we shall prove that a better upper bound on the maximum degree exists for  6-connected claw-free 1-planar graphs.  Moreover, there indeed exist 6-connected 1-planar graphs with maximum degree 8, see Fig. \ref{figa}.

\begin{thm}\label{thm:main_6connected}
Let $G$ be  a $6$-connected $1$-planar graph. If $G$ is claw-free, then $\Delta(G)\le 8 $.
\end{thm}

\begin{figure}[H]
  \centering
 \begin{tikzpicture}[scale=0.4,bezier bounding box]
 \tikzstyle{whitenode}=[fill=gray!10, draw=black, shape=circle, minimum size=0.2cm, inner
sep=1pt]

	\begin{pgfonlayer}{nodelayer}
		\node [style=whitenode] (9) at (-9.25, 7) {};
		\node [style=whitenode] (10) at (-2.25, 7) {};
		\node [style=whitenode] (11) at (-9.25, 0) {};
		\node [style=whitenode] (12) at (-2.25, 0) {};
		\node [style=whitenode] (13) at (-7.25, 5) {};
		\node [style=whitenode] (15) at (-4.25, 5) {};
		\node [style=whitenode] (16) at (-8.25, 4) {};
		\node [style=whitenode] (17) at (-7.25, 4) {};
		\node [style=whitenode] (18) at (-8.25, 1.5) {};
		\node [style=whitenode] (19) at (-4.25, 1.5) {};
		\node [style=whitenode] (20) at (2.75, 7) {};
		\node [style=whitenode] (21) at (9.75, 7) {};
		\node [style=whitenode] (22) at (2.75, 0) {};
		\node [style=whitenode] (23) at (9.75, 0) {};
		\node [style=whitenode] (24) at (4.75, 6) {};
		\node [style=whitenode] (25) at (7.75, 6) {};
		\node [style=whitenode] (26) at (4, 4.5) {};
		\node [style=whitenode] (27) at (4.75, 4.75) {};
		\node [style=whitenode] (28) at (4, 1.5) {};
		\node [style=whitenode] (29) at (7.75, 1.5) {};
		\node [style=whitenode] (30) at (5.5, 4) {};

	\end{pgfonlayer}
	\begin{pgfonlayer}{edgelayer}
		\draw [style={blackedge_thick}] (9) to (10);
		\draw [style={blackedge_thick}] (10) to (12);
		\draw [style={blackedge_thick}] (12) to (11);
		\draw [style={blackedge_thick}] (11) to (9);
		\draw [style={blackedge_thick}] (13) to (15);
		\draw [style=blueedge] (9) to (15);
		\draw [style=blueedge] (10) to (13);
		\draw [style=blueedge] (15) to (12);
		\draw [style={blackedge_thick}] (9) to (13);
		\draw [style={blackedge_thick}] (15) to (10);
		\draw [style=blueedge, bend left=75, looseness=2.00] (9) to (12);
		\draw [style=blueedge, bend left=285, looseness=2.00] (10) to (11);
		\draw [style={blackedge_thick}] (9) to (16);
		\draw [style={blackedge_thick}] (16) to (18);
		\draw [style={blackedge_thick}] (18) to (11);
		\draw [style={blackedge_thick}] (15) to (19);
		\draw [style={blackedge_thick}] (19) to (18);
		\draw [style={blackedge_thick}] (16) to (17);
		\draw [style={blackedge_thick}] (17) to (19);
		\draw [style={blackedge_thick}] (13) to (17);
		\draw [style=blueedge] (9) to (17);
		\draw [style=blueedge] (13) to (16);
		\draw [style=blueedge] (9) to (18);
		\draw [style=blueedge] (16) to (11);
		\draw [style=blueedge] (11) to (19);
		\draw [style=blueedge] (10) to (19);
		\draw [style=blueedge] (18) to (12);
		\draw [style=blueedge] (16) to (19);
		\draw [style=blueedge] (17) to (18);
		\draw [style=blueedge] (17) to (15);
		\draw [style=blueedge] (13) to (19);
		\draw [style={blackedge_thick}] (20) to (21);
		\draw [style={blackedge_thick}] (21) to (23);
		\draw [style={blackedge_thick}] (23) to (22);
		\draw [style={blackedge_thick}] (22) to (20);
		\draw [style={blackedge_thick}] (24) to (25);
		\draw [style=blueedge] (20) to (25);
		\draw [style=blueedge] (21) to (24);
		\draw [style=blueedge] (25) to (23);
		\draw [style={blackedge_thick}] (20) to (24);
		\draw [style={blackedge_thick}] (25) to (21);
		\draw [style=blueedge, bend left=75, looseness=2.00] (20) to (23);
		\draw [style=blueedge, bend left=285, looseness=2.00] (21) to (22);
		\draw [style={blackedge_thick}] (20) to (26);
		\draw [style={blackedge_thick}] (26) to (28);
		\draw [style={blackedge_thick}] (28) to (22);
		\draw [style={blackedge_thick}] (25) to (29);
		\draw [style={blackedge_thick}] (29) to (28);
		\draw [style={blackedge_thick}] (26) to (27);
		\draw [style={blackedge_thick}] (24) to (27);
		\draw [style=blueedge] (20) to (27);
		\draw [style=blueedge] (24) to (26);
		\draw [style=blueedge] (20) to (28);
		\draw [style=blueedge] (26) to (22);
		\draw [style=blueedge] (22) to (29);
		\draw [style=blueedge] (21) to (29);
		\draw [style=blueedge] (28) to (23);
		\draw [style=blueedge] (27) to (28);
		\draw [style=blueedge] (27) to (25);
		\draw [style=blueedge] (24) to (30);
		\draw [style=blueedge] (30) to (26);
		\draw [style=blueedge] (29) to (30);
		\draw [style=blueedge] (25) to (28);
		\draw [style={blackedge_thick}] (19) to (12);
		\draw [style={blackedge_thick}] (30) to (28);
		\draw [style={blackedge_thick}] (30) to (25);
		\draw [style={blackedge_thick}] (29) to (23);
		\draw [style={blackedge_thick}] (27) to (30);
	\end{pgfonlayer}
\end{tikzpicture}

  \caption{Two 6-connected claw-free  1-planar graphs with maximum degree 8}
  \label{figa}
\end{figure}

Every 1-planar graph has the minimum degree at most 7 and thus has connectivity  at most 7 \cite{Fabrici}. Furthermore, we shall prove that the connectivity of claw-free 1-planar graphs cannot reach the highest connectivity  7.

\begin{thm}\label{thm:connectivity}
Let $G$ be  a  $1$-planar graph. If $G$ is claw-free, then $\kappa(G)\le 6$.
\end{thm}

Note that the bound in Theorem~\ref{thm:connectivity} is also sharp. In the early preprint version of this paper, we were only aware that the graph $K_{2,2,2,2}$ and the other two graphs in Fig.~\ref{figa} are claw-free 1-planar graphs with connectivity~6. Very recently, Dr.~Sergey Pupyrev announced at the 33rd international symposium on graph drawing and network visualization that he constructed infinitely many claw-free 1-planar graphs with connectivity 6~\cite{Pupyrev}.

As a direct consequence of Theorem~\ref{thm:connectivity}, we have the following corollaries. 
\begin{cor}\label{thm:7_claw}
Every $7$-connected $1$-planar graph contains induced claws. 
\end{cor}


\begin{cor}
Every $7$-connected $1$-planar graph is not the line graph of any graph.
\end{cor}

\noindent \textbf{Outline of the paper.} In the remainder of this section, we introduce the necessary notations and
terminologies. Section \ref{sec:pre} presents  preliminary lemmas. 
Section~\ref{sec:thmproof} proves Theorems \ref{main0}, \ref{thm:main_6connected} and~\ref{thm:connectivity}. 
Section~\ref{sec:conclusion} concludes this article with some remarks and further open problems.


\noindent \textbf{Notations and terminologies.}
Let $G$ be a graph. The notation $\overline{G}$ denotes the \emph{complement} of  $G$. 
If $e=uv$ is an edge of $G$, then we say $u$ and $v$ are {\it adjacent}, $e$ and $u$ (and $v$) are {\it incident}. If $x$ is adjacent to $y$ in $G$, we write $x \sim_G y$; otherwise, $x \not \sim_G y$. When $G$ is clear from the context, we simply write $x \sim y$ (resp. $x \nsim y$). The {\it neighborhood} $N_G(u)$ of a vertex $u$ in $G$ is the set of all neighbors of $u$. The \emph{closed neighborhood} of $u$ is $N_G[u]:=N_G(u)\cup \{u\}$. A \emph{vertex-induced subgraph} of $G$ by a vertex subset $U \subseteq V(G)$, 
denoted by $G[U]$, is the subgraph of $G$ whose vertex set is $U$ and whose edge set 
consists of all edges of $G$ with both endvertices in $U$.   We denote by $C_n$,  $P_n$, and $K_n$ the cycle of order $n$, the path of order $n$, and the complete graph of order $n$, respectively.  A  \emph{$k$-cycle} refers to a cycle consisting of  $k$ edges.
 We define $\gamma_G(u, xyz)$ as a claw  of $G$, where $u$ is the unique vertex of degree three in the claw, and $x, y, z$ are its  vertices of degree one in the claw. The vertex $u$ is referred to as the \emph{center} of $\gamma_G(u, xyz)$.

A {\it drawing} of a graph $G = (V, E)$ is a mapping $D$ that
assigns to each vertex in $V$ a distinct point in the plane (or on
the sphere) and to each edge $uv$ in $E$ a continuous arc connecting $D(u)$ and
$D(v)$.  All 1-planar
drawings considered in this paper are such that no edge
crosses itself, no two edges cross more than once, and no two  edges incident with the same vertex
cross each other.
 Let $G$ be a graph  with a 1-planar drawing $D$.  The subdrawing $D|H$ of $H$ from $D$ is called a \emph{restricted drawing }of $D$. The interior region of a closed non-self-intersecting Jordan curve $\ell$ in the plane is denoted by $\ell_{int}$, and the exterior region by $\ell_{out}$. 
Sometimes we use the notation $x_1x_2\ldots x_i\ldots x_1$ to represent a closed curve, where each $x_i$ is either a vertex or a crossing point  of $D$, and the drawing of the curve is inherited from $D$. Moreover, for a closed non-self-intersecting Jordan curve $\ell$, if a vertex $u \in \ell_{int}$ and a vertex $v \in \ell_{out}$, we say that $u$ and $v$ are \emph{separated} by  $\ell$. 
The \emph{rotation} $\operatorname{rot}_D\left(u\right)$ of a vertex $u$ in the drawing $D$ is the cyclic permutation that records the (cyclic) counter-clockwise order in which the edges leave $u$. More formal,  we use the notation $\operatorname{rot}_D\left(u\right)=[v_1,v_2,\ldots, v_d]$  to denote the counter-clockwise order the edges incident with the vertex $u$ is $u v_1, u v_2,  \ldots, uv_d$.

%

\section{Preliminary lemmas}\label{sec:pre}
In this section, we present several lemmas. 
%

%
 



A \emph{bipartite graph} $G$ is a graph whose vertex set $V$ can be partitioned into two nonempty subsets $A$ and $B$ (i.e., $A \cup B = V$ and $A \cap B = \varnothing$) such that each edge of $G$ has one endvertex in $A$ and one endvertex in $B$.
The well-known  Mantel's theorem  states that any triangle-free  graph of order $n$ contains at most $\frac{n^2}{4}$ edges. In 1962, Erd\H{o}s \cite{Erdos}  showed the following result, which extends Mantel's theorem to nonbipartite graphs.
\begin{lem}[Erd\H{o}s  \cite{Erdos}]\label{erdos}
Let $G$ be a triangle-free graph of order $n$. If $G$ is not bipartite, then $e(G)\le \frac{1}{4}(n-1)^2+1$.
\end{lem}

In 2019, Ouyang, Ge and Chen \cite{ouyang2019} determined the following (tight) upper bound on the size of an $n$-vertex 1-planar graph $G$ with $\Delta(G)=n-1$.

\begin{lem}[Ouyang, Ge and Chen \cite{ouyang2019}, Remark 2]\label{ed_up}
Let $G$ be a $1$-planar graph of order $n\ge 3$. If $\Delta(G)=n-1$, then $e(G)\le 4n-9$.
\end{lem}

Next, we   describe local structures around a vertex in 4-, 6-, and 7-connected 1-planar graphs, respectively.  These structures depend only on connectivity and 1-planarity, not on the claw-free condition. These preparations are intended solely for Theorems~\ref{thm:main_6connected} and~\ref{thm:connectivity}; readers interested only in the proof of Theorem~\ref{main0} may skip the rest of this section.   Let $C$ be a cycle in a planar graph $G$.
We say that a cycle $C$ of $G$ is a \emph{separating cycle} if both the interior and the exterior of $C$ contain vertices of $G$.
Otherwise, $C$ is called a \emph{non-separating} cycle of $G$.  The \emph{associated plane graph} $G^{\times} $ of $G$ is the plane graph  obtained from $G$ by turning all crossings of $G$ into new vertices of degree four. A vertex in $G^{\times}$ is called \emph{fake} if it corresponds to some crossing of $G$, and is \emph{true} otherwise.

If $G^\times$ contains a 4-cycle $xyczx$ with exactly one fake vertex $c$, it is easy to verify that in $G$ this cycle can occur only in the two configurations shown in Fig.~\ref{4_cycle} (where $c$ is the unique fake vertex). We refer to the 4-cycle $xyczx$ in $G^\times$ arising from case (a) in Fig.~\ref{4_cycle} as a type-I 4-cycle.

\begin{figure}[H]
\centering
\begin{tikzpicture}[scale=0.6]
	\begin{pgfonlayer}{nodelayer}
		\node [style=whitenode] (0) at (-5, -1) {$x$};
		\node [style=whitenode] (1) at (-8, 1) {$y$};
		\node [style=whitenode] (2) at (-2, 1) {$z$};
		\node [style=whitenode] (3) at (-6, 4) {$e$};
		\node [style=whitenode] (4) at (-4, 4) {$d$};
		\node [style=pinknode] (5) at (-5, 3.25) {$c$};
		\node [style=whitenode] (6) at (5, -1) {$x$};
		\node [style=whitenode] (7) at (2, 2) {$y$};
		\node [style=whitenode] (8) at (8.25, 2) {$z$};
		\node [style=none] (9) at (5, 3) {};
		\node [style=none] (10) at (5, 0.75) {};
		\node [style=pinknode] (11) at (5, 2) {$c$};
		\node [style=none] (12) at (-5, -3) {(a)};
		\node [style=none] (13) at (5, -3) {(b)};
	\end{pgfonlayer}
	\begin{pgfonlayer}{edgelayer}
		\draw [style=blackedge] (1) to (0);
		\draw [style=blackedge] (2) to (0);
		\draw [style=blackedge] (4) to (1);
		\draw [style=blackedge] (3) to (2);
		\draw [style=blackedge] (7) to (6);
		\draw [style=blackedge] (8) to (6);
		\draw [style=blackedge] (7) to (8);
		\draw [style=blackedge, in=90, out=-90] (9.center) to (10.center);
	\end{pgfonlayer}
\end{tikzpicture}
\label{4_cycle}
\end{figure}

\begin{lem}\label{lem:4_connected_abxa}
Let $G$ be a 1-plane graph. Then the following statements hold.
\begin{itemize}
  \item[(i)] If $\kappa(G)\ge 4$, then  any 3-cycle of $G^\times$ containing exactly one fake vertex is non-separating;
  \item[(ii)] if $\kappa(G)\ge 6$, then  any type-I 4-cycle of $G^\times$  is non-separating; and
  \item[(iii)]  if $\kappa(G)= 7$, then  any 3-cycle of $G^\times$  is non-separating.
\end{itemize}
\end{lem}

\begin{proof}
The lemma follows directly from the 1-planarity and the well-known Menger's theorem(see page 210 in \cite{Bondy}).
\end{proof}

\begin{pro}\label{prop:4connectedlocal}
Let $G$ be a $4$-connected 1-plane graph, and let $u$ be a vertex of $G$. Let $u x$ and  $u y$ be two edges of $G$. Let $S$ denote the set of edges lying between $u x$ and $u y$ in $\operatorname{rot}_D(u)$.  If $|S| \geq 2$, then the edge $xy$ (if it exists) does not cross any edge in $S$.
\end{pro}
\begin{proof}
For   $uv\in S$, suppose that  $xy$ crosses $uv$ at a crossing $c$.  Without loss of generality, we assume that $uxyu$ bounds a finite region and $v\in (uxyu)_{out}$, shown in Fig. \ref{fig:Schematicdiagrams}(a). Moreover, since $|S| \ge 2$ and by 1-planarity, there exists an endvertex of an edge in $S$, distinct from $uv$, that lies in $(uxcu)_{int}$ or $(uycu)_{int}$.  Therefore  $uxcu$ or $uycu$ is a separating 3-cycle in $G^\times$, a contradiction to  Lemma \ref{lem:4_connected_abxa}(i). 
\end{proof}

\begin{pro}\label{prop:6connectedlocal}
Let $G$ be a $6$-connected 1-plane graph, and let $u$ be a vertex of $G$. Let $u x$ and  $uy$ be two edges of $G$.  Let $S=\{uv_1,uv_2,\ldots, uv_s\}$ denote the set of edges lying between $u x$ and $u y$ in $\operatorname{rot}_D(u)$. If $|S| \geq 2$ and $x\sim y$, then the edge $xy$ does not cross any edge incident with any vertex in $\{v_1,v_2, \ldots, v_s\}$.
\end{pro}

\begin{proof}
By Lemma \ref{prop:4connectedlocal}, $xy$ does not cross any edge in $S$. Thus we may assume that  $\{v_1,\ldots, v_s\}\subseteq (uxyu)_{int}$.   Now for some $v_i$ in $\{v_1,\ldots, v_s\}$ and $z(\ne u)$, suppose that   $xy$ crosses $v_iz$  at a crossing $c$.     Since $|S| \ge 2$ and by 1-planarity, there exists a vertex from $\{v_1, v_2, \ldots, v_s\}$, different from $v_i$, that lies in either $(uxcv_i u)_{int}$ or $(uycv_i u)_{int}$. Therefore  $uxcv_iu$ or $uycv_iu$ is a separating  type-I 4-cycle in $G^\times$, a contradiction to  Lemma \ref{lem:4_connected_abxa}(ii). Thus the proposition  holds.
\end{proof}

\begin{figure}
\centering
\begin{tikzpicture}[scale=0.45]
	\begin{pgfonlayer}{nodelayer}
		\node [style=whitenode] (0) at (-5, 0) {$u$};
		\node [style=whitenode] (1) at (-8, 3) {$x$};
		\node [style=whitenode] (2) at (-2, 3) {$y$};
		\node [style=whitenode] (3) at (-5, 4.25) {$v_i$};
		\node [style=none] (4) at (-5.5, 3.4) {$c$};
		\node [style=whitenode] (5) at (4, 0) {$u$};
		\node [style=whitenode] (6) at (1, 3) {$x$};
		\node [style=whitenode] (7) at (7, 3) {$y$};
		\node [style=whitenode] (8) at (3.5, 2) {$v_i$};
		\node [style=none] (9) at (3.25, 3.5) {$c$};
		\node [style=whitenode] (10) at (2.5, 4) {$z$};
        \node  at (-5, -2) {(a)};
        \node  at (4, -2) {(b)};
	\end{pgfonlayer}
	\begin{pgfonlayer}{edgelayer}
		\draw [blackedge](1) to (0);
		\draw [blackedge](0) to (2);
		\draw [blackedge](1) to (2);
		\draw [blackedge](3) to (0);
		\draw [blackedge](6) to (5);
		\draw [blackedge](5) to (7);
		\draw [blackedge](6) to (7);
		\draw [blackedge](8) to (5);
		\draw [blackedge](10.center) to (8);
	\end{pgfonlayer}
\end{tikzpicture}
\caption{Schematic diagrams for the proofs of Propositions  \ref{prop:4connectedlocal} and \ref{prop:6connectedlocal}}
\label{fig:Schematicdiagrams}
\end{figure}

\begin{pro}\label{prop:7connectedlocal}
Let $G$ be  a $7$-connected  $1$-plane graph.  Assume that $u$ is a vertex of $G$ and $\operatorname{rot}_G(u)=[u_1,u_2,\dots, u_k]$ where $k\ge 7$. Then the following statements hold.
\begin{itemize}
  \item [(i)]  If $3\le |i- j|\le k-3$, then $u_i\not\sim u_j$; and
  \item [(ii)] if $u_i\sim u_{i+2}$ where $u_{k+1}=u_1$ and $u_{k+2}=u_2$, then $u_iu_{i+2}$ crosses $uu_{i+1}$. 
\end{itemize}
\end{pro}
\begin{proof}
(i). Suppose that $u_i\sim u_j$.  Since $3 \le |i-j| \le k-3$, there are at least two edges lying between $uu_i$ and $uu_j$ in the clockwise direction, namely $uu_x$ and $uu_y$, and at least two edges lying between them in the counterclockwise direction, namely $uu_{x'}$ and $uu_{y'}$, as illustrated in Fig.~\ref{fig:schematic}(a).
 Let $S=\{uu_x,uu_y,uu_{x'},uu_{y'}\}$,  By the 1 -planarity of $G$, we consider the following two cases on the edge $u_i u_j$.
(1). $u_i u_j$ does not cross any edge in $S$. In this case, the vertices $u_x$ and $u_{x^{\prime}}$ are separated by the cycle $u u_i u_j u$, which contradicts Lemma \ref{lem:4_connected_abxa}(iii).
(2). $u_i u_j$ crosses an edge in $S$. We may assume that $u_i u_j$  crosses $u u_x$ at crossing $c$ shown in Fig. \ref{fig:schematic}(b), then  $u_y$ and $u_{x^{\prime}}$ are separated by the 3-cycle $uc u_j u$ in $G^\times$, contradicting Lemma \ref{lem:4_connected_abxa}(i).
Thus, (i) holds.

(ii). Suppose that $u_iu_{i+2}$ does not cross $uu_{i+1}$. If $u_iu_{i+2}$ is uncrossed, 
then $uu_iu_{i+2}u$ is a separating in $G$ since it separates $u_{i+1}$ and any vertex of $N_G(u)\setminus \{u_i,u_{i+1}, u_{i+2}\}$, which contradicts Lemma \ref{lem:4_connected_abxa}(iii). If $u_iu_{i+2}$ is crossed with an edge $st\ne uu_{i+1}$, we may assume that  where $s\ne u$ and $s\in (uu_iu_{i+2}u)_{int}$ and $t\in (uu_iu_{i+2}u)_{out}$. Then $uu_iu_{i+1}u$ is a separating 3-cycle in $G^\times$ which separates $u_{i+1}$ and $t$, which contradicts Lemma \ref{lem:4_connected_abxa}(iii).  Thus, (ii) holds.
\end{proof}

\begin{figure}[H]
\centering
\begin{tikzpicture}[scale=0.8]
	\begin{pgfonlayer}{nodelayer}
		\node [style=whitenode] (0) at (-3, 0) {$u$};
		\node [style=whitenode] (1) at (-3, 2) {$u_i$};
		\node [style=whitenode] (2) at (-3, -2) {$u_j$};
		\node [style=whitenode] (3) at (-1, 0.5) {$u_x$};
		\node [style=whitenode] (4) at (-1, -0.75) {$u_y$};
		\node [style=whitenode] (5) at (-5, 0.75) {$u_{x'}$};
		\node [style=whitenode] (6) at (-5, -0.75) {$u_{y'}$};
		\node [style=whitenode] (7) at (3, 0) {$u$};
		\node [style=whitenode] (8) at (3, 2) {$u_i$};
		\node [style=whitenode] (9) at (3, -2) {$u_j$};
		\node [style=whitenode] (10) at (5, 0.5) {$u_x$};
		\node [style=whitenode] (11) at (5, -0.75) {$u_{y}$};
		\node [style=whitenode] (12) at (1, 0.75) {$u_{x'}$};
		\node [style=whitenode] (13) at (1, -0.75) {$u_{y'}$};
		\node [style=none] (14) at (4.5, 0) {};
	\node [style=none] (c) at (4, 0) {$c$};
		\node [style=none] (15) at (-3, -3) {(a)};
		\node [style=none] (16) at (3, -3) {(b)};
	\end{pgfonlayer}
	\begin{pgfonlayer}{edgelayer}
		\draw [style=blackedge] (1) to (0);
		\draw [style=blackedge] (0) to (2);
		\draw [style=blackedge] (0) to (3);
		\draw [style=blackedge] (0) to (4);
		\draw [style=blackedge] (5) to (0);
		\draw [style=blackedge] (0) to (6);
		\draw [style=blackedge] (8) to (7);
		\draw [style=blackedge] (7) to (9);
		\draw [style=blackedge] (7) to (10);
		\draw [style=blackedge] (7) to (11);
		\draw [style=blackedge] (12) to (7);
		\draw [style=blackedge] (7) to (13);
		\draw [style=blackedge] (8) to (9);
		\draw [style=blackedge, in=165, out=-60] (8) to (14.center);
		\draw [style=blackedge, in=0, out=-15, looseness=2.25] (14.center) to (9);
	\end{pgfonlayer}
\end{tikzpicture}

\caption{A schematic diagram for the proof of Proposition \ref{prop:7connectedlocal} (i)}
\label{fig:schematic}
\end{figure}

\section{Proofs of main theorems  }\label{sec:thmproof}

%

\subsection{\textbf{Proof of Theorem \ref{main0}}}

%

Suppose to the contrary that there exists a vertex $v$ with $d_G(v)=k \ge 11$.  First, we assert that$\overline{G[N_G(v)]}$ is nonbipartite.
Suppose  that  $\overline{G[N_G(v)]}$  is bipartite.  Let the partite sets of $\overline{G[N_G(v)]}$ be $X$ and $Y$. Since $d_G(v)\ge 11$,  we have $|X| \ge 6$ or $|Y| \ge  6$. Then $G[X\cup \{v\}]$ or $G[Y\cup \{v\}]$ contains a subgraph isomorphic to the complete graph $K_7$ in $G$, a contradiction to the fact that $K_7$ is not 1-planar \cite{Czap2012}.  Thus $\overline{G[N_G(v)]}$ is  not bipartite.      Clearly, $\overline{G[N_G(v)]}$ is triangle-free.     Furthermore, by Lemma \ref{erdos}, we have 
\begin{align}\label{La}
e(\overline{G[N_G(v)]})\le \frac{1}{4}(k-1)^2+1.
\end{align}
On the other hand, the number of edges in \( G[N_G[v]] \) is given by
\begin{align}\label{Lb}
e(G[N_G[v]])= k+ \big(\binom{k}{2}-  e(\overline{G[N_G(v)]})\big)
\end{align}
Combining inequalities \ (\ref{La}) with  (\ref{Lb}),  we have 
\begin{align}\label{l1}
e(G[N_G[v]]) \ge   k+ \binom{k}{2}- (\frac{1}{4}(k-1)^2+1).
\end{align}
Since $d_{G[N_G[v]]}(v)=|N_G[v]|-1$ and clearly $|N_G[v]|\ge 12$, by Lemma \ref{ed_up},  we have
\begin{align}\label{l2}
e(G[N_G[v]])\le 4(k+1)-9.
\end{align}
Combining Inequalities (\ref{l1}) with (\ref{l2}),  we get $$k\le \lfloor 6+\sqrt{21} \rfloor =10,$$ which contradicts the assumption  $k\ge 11$.
Thus we have $\Delta(G)\le 10$.

Next, we construct a class of infinitely many claw-free 1-planar graphs with  maximum degree 10.
Let $G_1$ be the graph shown in Fig. \ref{fig:maximumdegree10}(i).
Furthermore, let $G_k$ be the graph obtained from $G_1$ and $P_k$ ($k \ge 2$) by identifying a vertex $v\ne u$ of $G_1$ with a degree-1 vertex  of $P_k$, as shown in Fig.~\ref{fig:maximumdegree10}(ii). Clearly, $G_k$ is 1-planar and $\Delta(G_k)=d_{G_k}(u)=10$. It can be directly verified that $G_k$ is claw-free, as desired.

%



%

\begin{figure}[H]
\centering
\begin{tikzpicture}[scale=0.55]
	\begin{pgfonlayer}{nodelayer}
		\node [style=whitenode] (0) at (-11, 5) {};
		\node [style=whitenode] (1) at (-14, 0) {$v$};
		\node [style=whitenode] (2) at (-8, 0) {$u$};
		\node [style=whitenode] (3) at (-5, 5) {};
		\node [style=whitenode] (4) at (-2, 0) {};
		\node [style=whitenode] (5) at (-11, 3) {};
		\node [style=whitenode] (6) at (-12, 1) {};
		\node [style=whitenode] (7) at (-10, 1) {};
		\node [style=whitenode] (8) at (-5, 3) {};
		\node [style=whitenode] (9) at (-6, 1) {};
		\node [style=whitenode] (10) at (-4, 1) {};
		\node [style=whitenode] (11) at (5, 5) {};
		\node [style=whitenode] (12) at (2, 0) {$v$};
		\node [style=whitenode] (13) at (8, 0) {$u$};
		\node [style=whitenode] (14) at (11, 5) {};
		\node [style=whitenode] (15) at (14, 0) {};
		\node [style=whitenode] (16) at (5, 3) {};
		\node [style=whitenode] (17) at (4, 1) {};
		\node [style=whitenode] (18) at (6, 1) {};
		\node [style=whitenode] (19) at (11, 3) {};
		\node [style=whitenode] (20) at (10, 1) {};
		\node [style=whitenode] (21) at (12, 1) {};
		\node [style=whitenode] (22) at (2, 1) {};
		\node [style=whitenode] (23) at (2, 2) {};
		\node [style=whitenode] (24) at (2, 4) {};
		\node [style=whitenode] (25) at (2, 5) {};
		\node [style=none] (26) at (-9, -2) {(i)};
		\node [style=none] (27) at (8, -2) {(ii)};
		\node [style=none] (28) at (1, 2.5) {$P_k$};
	\end{pgfonlayer}
	\begin{pgfonlayer}{edgelayer}
		\draw [style=blackedge] (0) to (2);
		\draw [style=blackedge] (2) to (1);
		\draw [style=blackedge] (1) to (0);
		\draw [style=blackedge] (3) to (2);
		\draw [style=blackedge] (2) to (4);
		\draw [style=blackedge] (4) to (3);
		\draw [style=blackedge] (3) to (8);
		\draw [style=blackedge] (8) to (9);
		\draw [style=blackedge] (9) to (10);
		\draw [style=blackedge] (10) to (8);
		\draw [style=blackedge] (9) to (2);
		\draw [style=blackedge] (10) to (4);
		\draw [style=blackedge] (0) to (5);
		\draw [style=blackedge] (5) to (6);
		\draw [style=blackedge] (6) to (7);
		\draw [style=blackedge] (7) to (5);
		\draw [style=blackedge] (6) to (1);
		\draw [style=blackedge] (2) to (7);
		\draw [style=blueedge] (5) to (1);
		\draw [style=blueedge] (6) to (0);
		\draw [style=blueedge] (0) to (7);
		\draw [style=blueedge] (5) to (2);
		\draw [style=blueedge] (2) to (6);
		\draw [style=blueedge] (7) to (1);
		\draw [style=blueedge] (2) to (10);
		\draw [style=blueedge] (4) to (9);
		\draw [style=blueedge] (9) to (3);
		\draw [style=blueedge] (8) to (2);
		\draw [style=blueedge] (10) to (3);
		\draw [style=blueedge] (8) to (4);
		\draw [style=blackedge] (11) to (13);
		\draw [style=blackedge] (13) to (12);
		\draw [style=blackedge] (12) to (11);
		\draw [style=blackedge] (14) to (13);
		\draw [style=blackedge] (13) to (15);
		\draw [style=blackedge] (15) to (14);
		\draw [style=blackedge] (14) to (19);
		\draw [style=blackedge] (19) to (20);
		\draw [style=blackedge] (20) to (21);
		\draw [style=blackedge] (21) to (19);
		\draw [style=blackedge] (20) to (13);
		\draw [style=blackedge] (21) to (15);
		\draw [style=blackedge] (11) to (16);
		\draw [style=blackedge] (16) to (17);
		\draw [style=blackedge] (17) to (18);
		\draw [style=blackedge] (18) to (16);
		\draw [style=blackedge] (17) to (12);
		\draw [style=blackedge] (13) to (18);
		\draw [style=blueedge] (16) to (12);
		\draw [style=blueedge] (17) to (11);
		\draw [style=blueedge] (11) to (18);
		\draw [style=blueedge] (16) to (13);
		\draw [style=blueedge] (13) to (17);
		\draw [style=blueedge] (18) to (12);
		\draw [style=blueedge] (13) to (21);
		\draw [style=blueedge] (15) to (20);
		\draw [style=blueedge] (20) to (14);
		\draw [style=blueedge] (19) to (13);
		\draw [style=blueedge] (21) to (14);
		\draw [style=blueedge] (19) to (15);
		\draw [style=blackedge] (22) to (12);
		\draw [style=blackedge] (23) to (22);
		\draw [style=blackedge] (25) to (24);
		\draw [style=dashed] (24) to (23);
	\end{pgfonlayer}
\end{tikzpicture}

\caption{Claw-free 1-planar graphs with maximum degree 10}
\label{fig:maximumdegree10}
\end{figure}

\subsection{\textbf{Proof of Theorem \ref{thm:main_6connected}}}\label{sec:proofofthmmain_6connected}

Suppose that $u$ is a vertex of $G$ with $d_G(u)\ge 9$. Let $D$ be a 1-planar drawing of $G$. We may assume that $N_G(u)=\{v_1,v_2,\ldots,v_{9}\}$ with $\operatorname{rot}_D(u)=[v_1,v_2,v_3,\ldots, v_9] $.
By the claw-freedom of  $G$ and noting the claw $\gamma_G(u,v_1v_4v_7)$, we have either $v_1\sim v_4$, $v_4\sim u_7$ or $v_1\sim v_7$.
By symmetry, we only consider $v_1\sim v_4$. By Proposition   \ref{prop:4connectedlocal}, we may assume that $v_2,v_3 \in (uv_1v_4u)_{int}$.
 By Proposition~\ref{prop:6connectedlocal}, we have $v_2 \nsim v_5$ and $v_2 \nsim v_8$. Together with the fact that $\gamma_G\left(u, v_2 v_5 v_8\right)$ is not induced, it follows that $v_5 \sim v_8$.
 Similarly, by Proposition   \ref{prop:4connectedlocal}, we may assume that  $v_6, v_7\in (uv_5v_8u)_{int}$. Note that $v_9\in (uv_1v_4u)_{out}$ and  $v_9\in (uv_5v_8u)_{out}$. By Proposition~\ref{prop:6connectedlocal}, we have $v_2\nsim v_6$ and $v_2\nsim v_9$, and $v_6\nsim v_9$. Then $\gamma_G\left(u, v_2 v_6 v_9\right)$  is induced, a contradiction. Thus, the theorem holds.

\subsection{\textbf{Proof of Theorem \ref{thm:connectivity}}}

Suppose that $\kappa(G)\ge 7$. It is known that $G$ has the minimum degree 7 \cite{Fabrici}. Then we may assume that $d_G(u)=7$ where $u\in V(G)$. Let $D$ be a 1-planar drawing of $G$. 
We may assume that $\operatorname{rot}_{D}(u)=[u_1, u_2, \ldots, u_{7}]$. By Proposition \ref{prop:7connectedlocal}(i), $u_1\nsim u_5$.  Since $\gamma_G(u,u_1u_3u_5)$ is a claw, by the claw-freedom of $G$, we have  $u_1\sim u_3$ or $u_3\sim u_5$. By symmetry, we may assume $u_1\sim u_3$ and by Proposition  \ref{prop:7connectedlocal}(ii), $u_1u_3$ crosses $uu_2$. Now by Proposition \ref{prop:7connectedlocal}(i), $u_4\nsim u_7$. By Proposition \ref{prop:7connectedlocal}(ii), if $u_2\sim u_7$, then it must cross $uu_1$, but this contradicts the 1-planarity of $D$. Therefore, $u_2\nsim u_7$.  Similarly, $u_2\nsim u_4$.
Thus $\gamma_G(u,u_2u_4u_7)$ is induced, a contradiction. Thus the theorem  holds.

%

%

\section{Concluding remarks}\label{sec:conclusion}
In this paper, we  give upper bounds on the maximum degree and connectivity of claw-free 1-planar graphs.  In the following, we present several open problems that remain unsolved.
Figure~\ref{f1} shows a 3-connected claw-free 1-planar graph with maximum degree 10. Taking a copy of $H_0$ we can obtain a new one by identifying the three vertices $x',y',z'$ in the copy with vertices $x,y,z$ in  $H_0$. Continuing this process, it is not difficult to construct infinitely many 3-connected claw-free 1-planar graphs with maximum degree 10. However, we have not yet found a 4-connected claw-free 1-planar graph with maximum degree  10.  We even conjecture that 
\emph{every 4-connected claw-free 1-planar graph has the maximum degree  at most 8.}  Note that if the conjecture holds, the upper bound 8 cannot be further improved, as we can construct a 4-connected claw-free 1-planar graph with maximum degree 8 shown in Fig.~\ref{f1}(ii). Based on this small graph, it is not difficult to construct infinitely many such graphs.

\begin{figure}[H]
\centering

\begin{tikzpicture}[scale=0.6, 
    every node/.style={draw, shape=circle, fill=gray!10, minimum size=0.2cm, inner sep=0pt},
    every edge/.style={-, draw=black, line width=1pt},
    blueedge/.style={draw=blue, line width=1pt}
]

\node (yp) at (210:3.3) {\tiny $y'$};
\node (zp) at (330:3.3) {\tiny $z'$};
\node (xp) at (90:3.1) {\tiny$x'$};
\node (sp) at (210:2.25) { \tiny$s'$};
\node (tp) at (330:2.25) {\tiny $t'$};
\node (u)  at (90:2.25) {\tiny $u$};
\node (s)  at (210:1.5) {\tiny $s$};
\node (t)  at (330:1.5) {\tiny $t$};
\node (y)  at (210:0.75) {\tiny $y$};
\node (z)  at (330:0.75) {\tiny $z$};
\node (x)  at (90:0.75)  {\tiny$x$};

\draw (xp) edge (u);
\draw (sp) edge (u);
\draw (sp) edge (tp);
\draw (tp) edge (u);
\draw (xp) edge (yp);
\draw (yp) edge (zp);
\draw (zp) edge (xp);
\draw (sp) edge (yp);
\draw (tp) edge (zp);
\draw (u)  edge (x);
\draw (x)  edge (y);
\draw (y)  edge (z);
\draw (z)  edge (x);
\draw (u)  edge (s);
\draw (y)  edge (s);
\draw (s)  edge (t);
\draw (u)  edge (t);
\draw (t)  edge (z);
\draw (s)  edge (sp);
\draw (t)  edge (tp);

\draw[blueedge] (xp) -- (sp);
\draw[blueedge] (yp) -- (u);
\draw[blueedge] (u)  -- (y);
\draw[blueedge] (s)  -- (x);
\draw[blueedge] (t)  -- (x);
\draw[blueedge] (u)  -- (z);
\draw[blueedge] (u)  -- (zp);
\draw[blueedge] (tp) -- (xp);
\draw[blueedge] (s)  -- (z);
\draw[blueedge] (y)  -- (t);
\draw[blueedge] (sp) -- (zp);
\draw[blueedge] (tp) -- (yp);
\end{tikzpicture}
\quad \quad
\begin{tikzpicture}[yshift=2cm, scale=0.28, 
    every edge/.style={-, draw=black, line width=1pt}]
 
	\begin{pgfonlayer}{nodelayer}
		\node [style=whitenode] (0) at (0, 11.5) {};
		\node [style=whitenode] (1) at (-8, -2) {};
		\node [style=whitenode] (2) at (8.25, -2) {};
		\node [style=whitenode] (3) at (-5, 0) {};
		\node [style=whitenode] (4) at (5.5, 0) {};
		\node [style=whitenode] (5) at (0, 7.5) {};
		\node [style=whitenode] (6) at (2.25, 1.75) {};
		\node [style=whitenode] (7) at (-2, 1.75) {};
		\node [style=whitenode] (8) at (0, 4) {};
        \node []  (a) at (0, -2.4) {};
	\end{pgfonlayer}
	\begin{pgfonlayer}{edgelayer}
		\draw (1) edge (0);
		\draw (0) edge (2);
		\draw (2) edge (1);
		\draw[blueedge] (0) to (3);
		\draw  (3) edge (4);
		\draw (4) edge (2);
		\draw (3) edge (1);
		\draw [blueedge] (1) to (4);
		\draw [blueedge] (3) to (2);
		\draw (0) edge (4);
		\draw (0) edge (5);
		\draw (5) edge (3);
		\draw [blueedge]  (5) to (4);
		\draw[blueedge]  (5) to (1);
		\draw (5) edge (6);
		\draw [blueedge] (6) to (3);
		\draw (6) edge (4);
		\draw [blueedge] (6) to (0);
		\draw [blueedge] (5) to (7);
		\draw (7) edge (3);
		\draw (6) edge (7);
		\draw [blueedge]  (7) to (4);
		\draw (5) edge (8);
		\draw  (8) to  (7);
		\draw (8) edge (6);
		\draw [blueedge] (8) to (3);
	\end{pgfonlayer}
\end{tikzpicture}
\caption{(i) A  claw-free 1-planar graph $H_0$ with connectivity 3 and maximum degree 10; (ii) a 4-connected claw-free  1-planar graph with maximum degree 8}
\label{f1}
\end{figure}


\begin{thebibliography}{99}
\footnotesize

\bibitem{Beineke}
L.W. Beineke, Characterizations of derived graphs, \emph{J. Combin. Theory} {\bf 9} (1970), 129--135.

\bibitem{Bekos}
M.A. Bekos, P. Bose, A. Büngener, V. Dujmovié, M. Hoffmann, M. Kaufmann, P. Morin, S. Odak,  A. Weinberger, On $k$-planar graphs without short cycles, arXiv:2408.16085, 2024. 
%





\bibitem{Bondy} J.A. Bondy, U.S.R. Murty, Graph Theory, GTM 244, Springer,
2008.


%


%


\bibitem{Chen}
 J. Chen, T. Wang,  H. Zhang, Acyclic chromatic index of triangle-free 1-planar graphs, \emph{Graphs and Combinatorics} \textbf{33} (2017), 859-868.

%
%
%
%
%
%


 
\bibitem{Czap2012}
 J. Czap, D. Hudák, 1-Planarity of complete multipartite graphs, \emph{Discrete Appl. Math.} \textbf{160} (2012) 505-512. 



%






\bibitem{Erdos}
 P. Erd\H{o}s, On a theorem of Rademacher-Turán, \emph{Illinois J. Math.} \textbf{6} (1962), 122-127.


\bibitem{Fabrici}
 I. Fabrici, T. Madaras, The structure of 1-planar graphs, \emph{Discrete Math.} \textbf{307} (2007), 854-865.

\bibitem{Faudree}
 R. Faudree, E. Flandrin, Z. Ryjáček, Claw-free graphs--a survey, \emph{Discrete Math.} \textbf{164} (1997), 87--147.


%
%



%
%


\bibitem{kobourov2017}
S.G. Kobourov, G.~Liotta, F.~Montecchiani, An annotated bibliography on
  1-planarity, \emph{Comput. Sci. Rev.} {\bf 25} (2017), 49--67.
%
%


%

\bibitem{Liang2022}
Z. Liang,  Total coloring of claw-free planar graphs, \emph{Discuss. Math. Graph Theory}, \textbf{42} (2022), 771-777.


\bibitem{Liang2016}
Z. Liang, E. Shan, L. Kang, Clique-perfectness of claw-free planar graphs, \emph{Graphs Combin.} \textbf{32} (2016), 2551--2562.



%


\bibitem{Matthews–Sumner}
 M. Matthews, D. Sumner, Hamiltonian results in $K_{1,3}$-free graphs, \emph{J. Graph Theory }\textbf{8} (1984), 139-146.
\bibitem{ouyang2019}
Z. Ouyang, J. Ge, Y. Chen, Remarks on the joins of 1-planar graphs, \emph{Appl. Math. Comput.} {\bf 362} (2019), 124537.

\bibitem{Plummer}
M.D. Plummer, Claw-free maximal planar graphs, \emph{Tech. rep., DTIC Document}, 1989, 9--23.

\bibitem{Plummer1994}
M.D. Plummer, Extending matchings in claw-free graphs, \emph{Discrete Math.} \textbf{125 }(1994), 301–
307

%
\bibitem{Pupyrev}
S. Pupyrev, Optimized One-Planarity Solver via SAT, 
first announced at the 33rd International Symposium on Graph Drawing and Network Visualization (GD2025), 
Session~04-02, 2025. Available at: \url{https://graphdrawing.github.io/gd2025/assets/pdfs/Session04-02-Sergey_Pupyrev.pdf}.





\bibitem{Ringel1965}
 G. Ringel, Ein Sechsfarbenproblem auf der Kugel, \emph{Abh. Math. Semin. Univ.
 Hambg.} {\bf 29} (1965), 107--117.

\bibitem{Shan2014}
E. Shan, Z. Liang and L. Kang, Clique-transversal sets and clique-coloring in planar graphs, \emph{European J. Combin.} {\bf 36} (2014), 367--376.



%




%

%







\end{thebibliography}
\end{document}